\documentclass[12pt, draft]{amsart}
\usepackage{latex_base}
\usepackage{macros}

\author{Janin Heuer}
\address{Janin Heuer, Technische Universit\"at Braunschweig, Institut f\"ur Analysis und Algebra, AG Algebra, Universit\"atsplatz 2, 38106 Braunschweig,
 Germany\medskip}
\email{janin.heuer@tu-braunschweig.de}

\author{Timo de Wolff}
\address{Timo de Wolff, Technische Universit\"at Braunschweig, Institut f\"ur Analysis und Algebra, AG Algebra, Universit\"atsplatz 2, 38106 Braunschweig,
 Germany\medskip}
\email{t.de-wolff@tu-braunschweig.de}
 


\title[]{Initial Application of SONC to Lyapunov Stability of Dynamical Systems}

\begin{document}

\begin{abstract}
	Certifying the stability of dynamical systems is a central and challenging task in control theory and systems analysis. 
	To tackle these problems we present an algorithmic approach to finding polynomial Lyapunov functions. 
	Our method relies on sums of nonnegative circuit functions (SONC), a certificate of nonnegativity of real polynomials. 
	We show that both the problem of verifying as well as the more difficult task of finding Lyapunov functions can be carried out via relative entropy programming when using SONC certificates.
	This approach is analogue yet independent to finding Lyapunov functions via sums of squares (SOS) certificates and semidefinite programming.
	Furthermore, we explore whether using the related, recently introduced DSONC certificate is advantageous compared to SONC for this type of problem.
	We implemented our results, and present examples to show their applicability.
\end{abstract}

\maketitle

\setlength{\abovedisplayskip}{5pt}
\setlength{\belowdisplayskip}{5pt}

\section{Introduction}

In this article we investigate dynamical systems which are modeled by a finite number of polynomial first-order differential equations. 
Specifically, we are interested in the stability of equilibrium points of such systems. 
This is most commonly characterized in terms of \struc{Lyapunov stability}, see \cref{def:stability}. 

Verifying stability is a problem that occurs in many areas of control theory and systems analysis \cite{KhalilNonlinearSystems,Hahn:Stability} as well as in applications such as population dynamics \cite{Stamova:Lyapunov}, chemical reactions \cite{Erdi:Mathematical,Feinberg:Existence}, aerospace guidance systems \cite{Park:Guidance,Roskam:Airplane}, and chaos theory \cite{Williams:Chaos}.
Determining the Lyapunov stability of equilibria is a highly challenging problem, since, in general, there do not exist systematic algorithms for finding Lyapunov functions.
For certain cases, searching for natural energy-based Lyapunov functions can be successful, see \cite{KhalilNonlinearSystems}. 
But finding exact energy functions can also be a difficult task, and the property of having energy may not even exist for, e.g., economic, biological, or abstract mathematical systems. 

In this paper we give an algorithmic approach to finding Lyapunov functions for polynomial dynamic systems.
Due to Lyapunov's stability theorem \cref{thm:lyapunovstability}, we know that deciding stability is closely related to deciding nonnegativity of real multivariate polynomials, which is, however, a hard problem (both theoretically and practically).
Thus, assuming the existence of polynomial Lyapunov functions, it is reasonable to relax the problem by testing whether the polynomial at hand certifies a certain (easier to test) certificate of nonnegativity.
Existing approaches using this connection rely on the \struc{sum of squares (SOS)} certificate, see e.g. \cite{PapachristodoulouAdvancesSOS,PapachristodoulouPrajna:LyapunovSOS,Hafstein:Lyapunov,Tan:Nonlinear}. 
Obviously, a polynomial which admits a representation as a sum of squares of other polynomials is necessarily nonnegative.
Finding such an SOS representation, however, requires only to solve a semidefinite program, a well-known class of convex optimization problems.
For an overview on SOS, see e.g. \cite{Laurent:Survey,Lasserre:GlobalOpt,Parrilo:Thesis}.

Another certificate of nonnegativity is the \struc{sum of nonnegative circuits (SONC)} certificate. 
The first study of circuit polynomials was conducted for the special case called \struc{agiforms} by Reznick in \cite{Reznick:AGI}. 
The general polynomial formulation was then given by Iliman and the second author in \cite{Iliman:deWolff:Circuits}. 
In the signomial setting, the corresponding theory was developed independently shortly after in \cite{Chandrasekaran:Shah:SAGE-REP}.
The theory was also independently developed in context of chemical reaction networks in \cite{pantea-jac}.
A circuit polynomial is mainly characterized by the fact that its support forms a minimally affine dependent set. 
It also satisfies a number of other conditions which we formally present in \cref{def:circuitpolynomial}.
In order to decide whether a single circuit polynomial is nonnegative, one only needs to solve a system of linear equations. 
This was proven by Iliman and the second author and can be motivated as a consequence of the \struc{arithmetic mean / geometric mean (AM/GM) inequality}.
This approach to polynomial nonnegativity is independent of SOS and is especially suited for large, sparse polynomials. 
Certifying nonnegativity using the SONC approach can be done via relative entropy programming, which is also a convex optimization problem; see \cite{Iliman:deWolff:FirstGP,Chandrasekaran:Shah:SAGE-REP,Murray:Chandrasekaran:Wierman:SigOptREP}.

\medskip

In this paper, we show that certificates of nonnegativity can be used to prove Lyapunov stability of polynomial dynamical systems, see \cref{thm:gen_stability}. 
While this holds for arbitrary certificates of nonnegativity, we especially focus on the case of the SONC cone. 
We further investigate how useful the \struc{DSONC} cone, an object recently introduced in \cite{Heuer:deWolff:DualityOfSONC}, is for this type of problem, see \cref{thm:dsonc_stability}.
This is a subcone of the SONC cone which is derived from the dual of the SONC cone and has the added advantage that verifying containment of a given polynomial in this cone can be realized via linear programming. 

While the first part of \cref{sec:PolyLyap} solely focuses on \emph{verifying} whether a given function is a Lyapunov function, \cref{sec:StabilitySONC} deals with the more difficult question of actually \emph{finding} a Lyapunov function of a given system. 
We prove in \cref{prop:LyapSONCREP} that despite the added difficulty of not having a concrete Lyapunov candidate with which to verify stability, searching for a Lyapunov function still only requires relative entropy programming. 
We present an algorithm realizing this in \cref{Alg:primalREPalg}. 
In \cref{sec:StabilityDSONC} we show that if we attempt to search for a Lyapunov function using the DSONC approach, then we can no longer to this via linear programming. 
Instead, we prove in \cref{prop:LyapDSONCREP} that Lyapunov search via the DSONC cone also requires solving a relative entropy program, which means that using the DSONC cone in this case no longer provides a computational advantage.
In \cref{sec:Examples} we present selected examples which illustrate the functionality of our algorithm. 
We close the article by providing an outlook detailing possible next steps to further develop this line of research.

\section*{Acknowledgments}

We thank Christian Kirches for his helpful comments. 
The authors were supported by the DFG grant WO 2206/1-1. 

\section{Preliminaries}

\subsection{Notation}

Throughout this article we use the following basic notations.
We refer to the sets of natural numbers, and real numbers via the symbols $\struc{\N}$, and $\struc{\R}$, respectively.
If we restrict $\R$ to the positive orthant, then we write $\struc{\R_{>0}} = \{ x \in \R \ : \ x > 0 \}$.
If we refer to vectors, then we use bold symbols. 
E.g., we write the vector $(x_1, \ldots, x_n) \in \R^n$ as $\struc{\xb}$. 
For a given finite set $S \subseteq \R^n$ we denote by $\struc{\conv(S)}$ the \struc{convex hull} of $S$, and we refer to the set of \struc{vertices} of $\conv(S)$ by $\struc{\vertices{\conv(S)}}$. 
The \struc{cardinality} of $S$ is denoted by $\struc{\# S}$.
If $C$ is a given cone in $\R^n$, or a linear space, then we write $\struc{\check{C}}$ for the \struc{dual cone/space}.

\subsection{Nonlinear Systems Analysis}

The objective of stability theory is to analyze the behavior of solutions of systems of differential equations without the need to compute their solution trajectories explicitly.
While stability was studied by Lagrange as early as 1788 \cite{Lagrange:Mecanique}, the basic definitions as well as many fundamental theorems were introduced by Lyapunov in 1892 in his doctoral dissertation, see \cite{Lyapunov:1992} for a published and translated version of the original Russian text.
In this subsection we give a brief overview about the subject. 
For a more comprehensive introduction see, e.g., \cite{KhalilNonlinearSystems}, \cite{VidyasagarNonlinearSystems}, or \cite{SastryNonlinearSystems}.\\

\medskip

In this paper we study the stability of polynomial systems of differential equations of the form
\begin{align}
		\begin{aligned}
			& & \dot{\xb}(t) & \ = \ & \Vector{f}(\xb(t)), \\
			& \text{subject to } & \xb(t_0) & \ = \ & \xb_0.
		\end{aligned}
	\label{eq:nonlinearsystem}
\end{align}
Here, $\struc{\xb(t)} \in \R^n$ denotes the \struc{state} of the system in $t \ge 0$ with corresponding vector of derivatives $\struc{\dot{\xb}(t)}$, $\struc{\Vector{x_0}} \in \cU$ is the \struc{initial state} for some neighborhood $\cU \subseteq \R^n$ of the origin, and $\Vector{f}$ is a vector of real polynomial functions $f_i: \cU \rightarrow \R^n$, $i \in [n]$.
Note that solutions to \cref{eq:nonlinearsystem} exist and are unique (locally), since $\Vector{f}$ is locally Lipschitz in $\cU$. 
This is a direct consequence of the Picard--Lindel\"of Theorem; see, e.g., \cite{Hale:ODE}.

\begin{definition}
	A point $\struc{\xb^\ast} \in \R^n$ is called an \struc{equilibrium point} of \cref{eq:nonlinearsystem}, if $\Vector{f}(\xb^\ast(t)) = \ob$ for all $t \ge 0$.
\end{definition}

In other words, an equilibrium point $\xb^\ast \in \R^n$ is a constant solution to the system \cref{eq:nonlinearsystem}. 
Throughout the paper we assume without loss of generality that $\xb^\ast = \ob$ is an equilibrium point.
In Lyapunov stability theory, one is interested in the behavior of a system's trajectories for initial states that are close to an equilibrium.

\begin{definition}
	An equilibrium $\xb^\ast \in \R^n$ of the system \cref{eq:nonlinearsystem} is called \struc{stable}, if for every $\epsilon > 0$ there exists some $\delta(\epsilon) > 0$ such that $\norm{\Vector{x_0} - \xb^\ast} < \delta(\epsilon) \ \Rightarrow \ \norm{\xb(t) - \xb^\ast} < \epsilon$ for all $t \ge 0$. 
	An equilibrium is called \struc{asymptotically stable} if it is stable and $\delta$ can be chosen such that $\norm{\Vector{x_0}} < \delta \ \Rightarrow \ \lim_{t \to \infty} \xb(t) = \xb^\ast$.
	If an equilibrium is not stable, then it is called \struc{unstable}.
\label{def:stability}
\end{definition}

A common way of analyzing the stability of a system is known as Lyapunov's Direct Method, which we present in the following theorem. 

\begin{theorem}
	\label{thm:lyapunovstability}
	The equilibrium $\ob = \xb^\ast \in \cU \subset \R^n$ of the system \cref{eq:nonlinearsystem} is stable, if there exists a continuously differentiable function $V : \cU \rightarrow \R$ such that
	\begin{align*}
		V(\ob) & \ = \ 0, \\
		V(\xb) & \ > \ 0 \qquad \text{ for all } \xb \in \cU \setminus \{ \ob \}, \\
		\dot{V}(\xb) & \ \le \ 0 \qquad \text{ for all } \xb \in \cU.
	\end{align*}
	If furthermore
	\begin{align*}
		\dot{V}(\xb) \ < \ 0 \qquad \text{ for all } \xb \in \cU \setminus \{ \ob \},
	\end{align*}
	then $\xb^\ast$ is asymptotically stable.
\end{theorem}

The function $V$ used in this theorem can be interpreted as a function which behaves similarly to the function describing the total energy stored in a physical system. 
That is, if we take $V$ to be a function describing the (positive) amount of energy present in a system, and if this energy decreases over time ($-\dot{V} \le 0$), then the system is considered stable if it reaches a final resting state ($V(\ob) = 0$).
Lyapunov's contribution is showing that we can prove whether an equilibrium is stable, asymptotically stable, or unstable, without knowing the exact physical energy. 
Instead, we can just consider a function that behaves similarly.
This function is called the \struc{Lyapunov function}. \\

The goal of this article is finding a new algorithm that searches for such polynomial Lyapunov functions using methods from the area of polynomial nonnegativity.
In the next subsection we see that, using \cref{thm:lyapunovstability} as a starting point, polynomial optimization is a natural way of finding polynomial Lyapunov functions.

\subsection{Polynomial Optimization}

We give a brief overview about the fundamentals of polynomial optimization and the closely related problem of deciding nonnegativity of polynomials.
For a comprehensive introduction we refer the interested reader to, e.g., \cite{Laurent:Survey,Lasserre:IntroductionPolynomialandSemiAlgebraicOptimization,Blekherman:Parrilo:Thomas}. \\

\medskip

A real polynomial $p$ in $\xb \in \R^n$ is a function of the form
\begin{align}
    p(\xb) \ = \ \sum_{\alpb \in A} c_{\alpb}\xb^{\alpb},
\label{eq:polynomial}
\end{align}
where the set $A \subset \N^n$ of exponents is finite, and we use the multi-index notation $\struc{\xb^{\alpb}} \ = \ \prod_{i = 1}^n x_j^{\alpha_j}$. 
The \struc{coefficients} $c_{\alpb}$ are in $\R$ for every exponent $\alpb \in A$.
We call the set $\struc{\supp(p)} = \{\alpb \in A \ : \ c_{\alpb} \ne 0\}$ the \struc{support} of $p$.
The convex hull $\conv(\supp(p))$ of the support set of $p$ is called the \struc{Newton polytope} of $p$.
Let 
\begin{align*}
	\struc{\polys} \ = \ \spann_{\R} \set{\xb^{\alpb} \ : \ \alpb \in A}
\end{align*}
denote the set of real polynomials with fixed support set $A$.
Note that every polynomial in $\polys$ is uniquely described by its \struc{coefficient vector} $\struc{(c_{\alpb})_{\alpb \in A}}$, which is an element in $\R^{\# A}$.
This means that we can find an isomorphism $\polys \simeq \R^{\# A}$.
For this reason, we use the notation $\polys$ for both the set of polynomials supported on $A$, and the set of vectors indexed by elements in $A$.
Whenever we do not restrict the support sets of our polynomials, i.e., when we refer to the space of all real multivariate polynomials, we write $\struc{\R[\xb]}$. 
We say that a given polynomial $p$ is \struc{nonnegative} if $p(\xb) \ge 0$ for all $\xb \in \R^n$.


\begin{definition}
    Let $p$ be a polynomial in $\R[\xb]$. 
    The problem
    \begin{align}
        \min_{x \in \R^n} \ p(\xb) 
        \label{eq:CPOP}
        \tag{POP}
    \end{align}
    is called an \struc{(unconstrained) polynomial optimization problem (POP)} with \struc{objective function} $p$.
\label{def:cpop}
\end{definition}



In general, POPs are not convex and solving them is NP-hard, see e.g. \cite{Laurent:Survey}. 

Observe that the problem \cref{eq:CPOP} can be equivalently stated as \struc{nonnegativity} problem 
\begin{align}
	\max \set{\gamma \in \R \ : \ p(\xb) - \gamma \ge 0 \ \text{for all } \xb \in \R^n}.
	\label{eq:initialproblem}
\end{align}
Specifically, we are interested in the \struc{nonnegativity cone with respect to $A$} 
\begin{align*}
	\struc{\ANonnegCone} = \set{ p \in \R[\xb] \ : \ \supp{(p)} \subseteq A, \ p(\xb) \ge 0 \ \text{ for all } \xb \in \R^n},
\end{align*}
where $A$ is a finite subset of $\N^n$.

The reformulation \cref{eq:initialproblem} allows us to use objects called \struc{certificates of nonnegativity} to find lower bounds to \cref{eq:CPOP}.
These are conditions on real multivariate polynomials which imply nonnegativity, but are easier to test than nonnegativity itself.
We further require that these certificates hold for a reasonably large class of polynomials.
The classical example of a certificate of nonnegativity are \struc{sums of squares (SOS)}
\begin{align*}
	\struc{\Sigma_A} \ = \ \set{p \in \R[\xb] \ : \ \supp(p) \subseteq A, \ p(\xb) = \sum_{i = 1}^k s_i^2(\xb) \text{ for some } s_i \in \R[\xb]} \subseteq \ANonnegCone.
\end{align*}
For more on SOS, see e.g. \cite{Laurent:Survey,Lasserre:GlobalOpt}.

\subsubsection{The SONC Cone}
A more recent approach to nonnegativity is a class of polynomials called \struc{circuit polynomials}.
These polynomials give rise to a new certificate of nonnegativity called the \struc{sum of nonnegative circuits (SONC)} certificate.
The idea of circuit polynomials was first introduced by Reznick in \cite{Reznick:AGI}. 
Reznick investigates a special form of circuit polynomials which he calls \struc{simplicial agiform}.
The general definition was given by Iliman and the second author in \cite{Iliman:deWolff:Circuits}.
Recall that a set $A$ is called a \struc{circuit} if $A$ is minimally affine dependent, i.e., if all real subsets of $A$ are affinely independent; see, e.g., \cite{DeLoera:et:al:Triangulations}.

\begin{definition}[Circuit Polynomial]
	A polynomial $p \in \R[\xb]$ with finite support set $A \subset \N^n$ is called a \struc{circuit polynomial} if 
	\begin{enumerate}
		\item $A$ is a simplicial circuit, 
		\item it satisfies 
			\begin{align*}
				\alpb \in \vertices{\conv(A)} \ \Rightarrow \ \alpb \in (2\N)^n \ \text{ and } \ c_{\alpb} > 0, \text{ and}
			\end{align*}
		\item if $\betab \in A \setminus \vertices{\conv(A)}$, then $\betab$ is a point in the strict interior of $\conv(A)$.
	\end{enumerate}
\label{def:circuitpolynomial}
\end{definition}

Note that \cref{def:circuitpolynomial} implies that circuit polynomials are of the form
\begin{align*}
	p(\xb) = \sum_{\alpb \in A^+} c_{\alpb} \xb^{\alpb} + c_{\betab}\xb^{\betab},
\end{align*}
where $\struc{A^+} = \set{\alpb \in A \ : \ \alpb \in (2\N)^n, c_{\alpb} > 0} \subset (2\N)^n$ and $\betab$ is contained in the strict interior of the Newton polytope $\conv(A)$ of $p$. 
Since $\conv(A)$ is, by definition, a simplex, there exist unique barycentric coordinates of $\betab$ with respect to $A^+$.
We denote these by $\struc{\Vector{\lambda}}$.
Iliman and the second author further show that nonnegativity of circuit polynomials can be decided by an invariant called the \struc{circuit number}.

\begin{theorem}[\cite{Iliman:deWolff:Circuits}]
	A circuit polynomial $p(\xb) = \sum_{\alpb \in A^+} c_{\alpb} \xb^{\alpb} + c_{\betab}\xb^{\betab}$ is nonnegative if and only if 
	\begin{enumerate}
		\item $p$ is a sum of monomial squares, i.e., $\betab \in (2\N)^n$ and $c_{\betab} > 0$, or 
		\item $|c_{\betab}| \ \le \ \prod\limits_{\alpb \in A^+} \left( \frac{c_{\alpb}}{\lambda_{\alpb}}\right)^{\lambda_{\alpb}} \ =: \ \struc{\Theta_p}$.
	\end{enumerate}
	The invariant $\Theta_p$ is called the \struc{circuit number} of $p$.
\label{thm:nonnegativecircuit}
\end{theorem}

To utilize the above results on circuit polynomials for polynomials with general support sets, we define the cone of \struc{sums of nonnegative circuit (SONC) polynomials}.

\begin{definition}[SONC Cone]
	The \struc{SONC cone} $\struc{\ASONC}$ is the subset of all nonnegative polynomials $p$ with finite support in $A \subset \N^n$, which admit a representation as sums of nonnegative circuit polynomials or monomial squares.
\label{def:sonccone}
\end{definition}

\begin{theorem}[\cite{deWolffPositivstellensatz}]
	Let $A \subseteq \N^n$ be some finite support set. 
	Then $\ASONC$ is a convex, full-dimensional cone in $\ANonnegCone$. 
	In general, it further holds that $\ASONC \not\subset \Sigma_A$.
\end{theorem}


In polynomial optimization, we make use of the fact that being SONC implies being nonnegative by observing that 
\begin{align*}
	\max\{ \gamma \in \R^n \ : \ p(\xb) - \gamma \ge 0 \text{ for all } \xb \in \R^n\} \ \ge \ \max\{ \gamma \in \R^n \ : \ p(\xb) - \gamma \text{ is SONC}\}.
\end{align*}
I.e., we optimize over the SONC cone to find lower bounds to our initial polynomial optimization problem.
The circuit-based approach to SONC optimization can be carried out via \struc{geometric programming}, see, e.g., \cite{Iliman:deWolff:FirstGP}.
In the context of polynomial optimization, we are usually given information about the coefficient vector. 
Therefore it makes sense to assume a fixed sign distribution for the coefficient vector.
Let $p \in \R^A$ be a polynomial of the form $p(\xb) = \sum_{\alpb \in A} c_{\alpb} \xb^{\alpb}$. 
Then we write the support of $p$ as $A = A^+ \cup A^-$, where $A^+$ is defined as before, that is, as the set of even exponent vectors with positive corresponding coefficients, and $\struc{A^-} = A \setminus A^+$.
In the case of the SONC cone we indicate this by writing $\struc{\signedSONC}$. 
This kind of separation into a \struc{positive and negative support} is applied often when approaching polynomial optimization problems, see e.g. \cite{Dressler:Iliman:deWolff:FirstConstrained, Iliman:deWolff:FirstGP, Murray:Chandrasekaran:Wierman:NewtonPolytopes, Murray:Chandrasekaran:Wierman:SigOptREP}.

In this paper we make use of an alternative characterization of the SONC cone, which is independent of an explicit decomposition into circuits, and was first introduced by Chandrasekaran and Shah in \cite{Chandrasekaran:Shah:SAGE-REP} under the name \struc{sums of arithmetic-geometric exponentials (SAGE)}.
For equal support sets, the SONC and SAGE cones are equivalent, see \cite{Wang:supports,Murray:Chandrasekaran:Wierman:NewtonPolytopes}; see also \cite{Forsgaard:deWolff:BoundarySONCCone}. 
In the SAGE language, results are stated for the case of \struc{exponential sums} $f(\xb) = \sum_{\alpb \in A} c_{\alpb} \exp(\xb\T \alpb)$.
If we require $A \subset \N^n$ we can recover polynomials on the positive orthant by using the variable transformation $x_i \mapsto \ln(x_i)$ for all $i \in [n]$, see also \cite{Iliman:deWolff:Circuits}.
%
%
Using results from the SAGE literature, we obtain the following characterization of the SONC cone in the language of polynomials. 

\begin{theorem}[{Essentially \cite[Lemma 2]{Chandrasekaran:Shah:SAGE-REP} and \cite[Theorem 3]{Murray:Chandrasekaran:Wierman:NewtonPolytopes}}]
	Let $A = A^+ \cup A^- \subset \N^n$ be a finite, nonempty support set.
	Then a polynomial $p(\xb) = \sum_{\alpb \in A^+} c_{\alpb} \xb^{\alpb} + \sum_{\betab \in A^-} c_{\betab} \xb^{\betab}$ is contained in $\signedSONC$ if and only if for every $\betab \in A^-$ there exist $\Vector{c}^{(\betab)}$ and $\Vector{v}^{(\betab)}$ in $\R^{A^+}_{\ge 0}$ such that
	\setlength{\abovedisplayskip}{0pt}
	\setlength{\belowdisplayskip}{0pt}
	\begin{align*}~
		\begin{aligned}~
			\sum_{\betab \in A^-} c_{\alpb}^{(\betab)} &\ = \ c_{\alpb} &\text{ for all } \ \alpb \in A^+ , & \\
		D(\Vector{v}^{(\betab)}, e \cdot \Vector{c}^{(\betab)}) &\ \le \ c_{\betab} &\text{ for all } \ \betab \in A^-, &\text{ and } \\
		\sum_{\alpb \in A^+} v_{\alpb}^{(\betab)} (\alpb - \betab) &\ = \ 0  &\text{ for all } \ \betab \in A^-, &
		\end{aligned}
	\end{align*}
	where $\struc{D}: \polys_{\ge 0} \times \polys_{\ge 0} \rightarrow \R$ is the \struc{relative entropy function}
		$D(\Vector{v}, \Vector{c}) \ = \ \sum_{\alpb \in A} v_{\alpb} \ln \left(\frac{v_{\alpb}}{c_{\alpb}}\right)$.
	\label{thm:sagecontainment}
\end{theorem}

If we use the SAGE based approach for polynomial optimization, then this requires solving a \struc{relative entropy program}, see \cite{Chandrasekaran:Shah:SAGE-REP}.

\begin{definition}[Relative Entropy Program]
	An optimization problem is called a \struc{relative entropy program (REP)}, if it is of the form
	\begin{align*}
		\min \qquad &f(\xb) \\
		\text{subject to } \qquad &g_i(\xb) \ \ge \ 0 \qquad \text{ for all } \ i \in [m], \\
		&\xb \in \cR_n,
	\end{align*}
	where $f, g_1, \ldots, g_m$ are linear functions and 
	\begin{align*}
		\struc{\cR_n} = \set{(\yb, \zb, \Vector{t}) \in \R^n \times \R^n \times \R^n \ : \ y_j \ln \left( \frac{y_j}{z_j}\right) \le t_j \text{ for all } j \in [n]}
	\end{align*}
	denotes the \struc{relative entropy cone}.
\end{definition}

Finally, we also point out that various results in the realm of SONC and SAGE, e.g. \cref{thm:nonnegativecircuit}, were shown independently by Craciun, Pantea, and Koeppl already in 2012 in the context of chemical reaction networks \cite{pantea-jac}.

\subsubsection{The DSONC Cone}
\label{sec:DualSONC}

In this article, we focus on the \emph{dual} of the SONC cone, and the closely related \struc{DSONC cone} as a secondary certificate of nonnegativity.
The dual of the SONC cone has been studied in e.g. \cite{Dressler:Heuer:Naumann:deWolff:DualSONCLP,Katthaen:Naumann:Theobald:UnifiedFramework}, and the DSONC cone was recently introduced in \cite{Heuer:deWolff:DualityOfSONC}. 
We only give a brief overview here, for more details we refer the interested reader to \cite{Heuer:deWolff:DualityOfSONC}. \\

The DSONC cone not only yields a certificate of nonnegativity, it also has the added advantage that optimizing over this cone can be done via \emph{linear} programming, see \cref{prop:dualfeasibilty}.
In order to define the DSONC cone, we 
use the \struc{natural duality pairing}
\begin{align*}
	\struc{\Vector{v}(p)} \ = \ \sum_{\alpb \in A^+} v_{\alpb} c_{\alpb} +\sum_{\betab \in A^-} v_{\betab} c_{\betab} \ \in \ \R,	
\end{align*}
where $\Vector{v}(\cdot) \in \dualpolys$ and $p(\xb) = \sum_{\alpb \in A^+} c_{\alpb} \xb^{\alpb} + \sum_{\betab \in A^-} c_{\betab} \xb^{\betab} \in \polys$. 
With this and the usual definition of dual cones, the dual SONC cone is the set 
\begin{align*}
	\dualsignedSONC = \set{ \Vector{v} \in \dualpolys \ : \  \Vector{v}(p) \ge 0 \text{ for all } p \in \signedSONC}.
\end{align*}


To associate polynomials $p \in \polys$, where $A = A^+ \cup \set{\betab}$, with elements $\Vector{v} \in \dualcircuitSONC$ we consider
\begin{align}
	p(\xb) \ = \ \sum_{\alpb \in A^+} v_{\alpb} \xb^{\alpb} + v_{\betab} \xb^{\betab}.
\label{eq:identification}
\end{align}
Every polynomial of this form is also contained in the SONC cone $\circuitSONC$ and thus nonnegative, see \cite{Dressler:Heuer:Naumann:deWolff:DualSONCLP}. 
Analogously to the construction of the primal SONC cone $\signedSONC$, we now define the following generalization to arbitrary support sets. 

\begin{definition}[DSONC Cone]
	Let $A = A^+ \cup A^-$ be an arbitrary fixed support set. 
	Then the Minkowski sum
	\begin{align*}
		\struc{\signedDSONC} \ = \ \sum_{\betab \in A^-} \set{ \sum_{\alpb \in A^+} c_{\alpb} \xb^{\alpb} + c_{\betab} \xb^{\betab} \ : \ \Vector{c} \in \dualcircuitSONC}.
	\end{align*}
	is called the \struc{(signed) DSONC cone}. 
	If we do not fix the sign distribution, then we write 
	$\struc{\ADSONC}$.
\end{definition}

We point out that this cone can trivially be generalized to non-fixed support sets as well by writing
\begin{align}
	\struc{\DSONC} \ = \ \set{ p \in \R[\xb] \ : \ \text{ there exists } A \subset \N^n \ \text{ such that } \ p \in \ADSONC}. \label{Equation:DSONCCone}
\end{align}
The following feasibility problem verifies containment in $\dualcircuitSONC$ and, by extension, the DSONC cone $\circuitDSONC$.

\begin{proposition}[\cite{Dressler:Heuer:Naumann:deWolff:DualSONCLP}]
		For some fixed support set $A = A^+ \cup \set{\betab} \subset \N^n$, let $\Vector{v} \in \dualpolys$ and $v_{\alpb} \ge 0$ for every $\alpb \in \vertices{\conv(A)}$.
		The linear feasibility program
		\begin{align}
			\ln\left(\frac{|v_{\betab}|}{v_{\alpb}}\right) \ \le \ (\alpb - \betab)\T \Vector{\tau}^{(\betab)} \;\text{ for all } \; {\alpb} \in A^+ 
		\label{eq:linprg}
		\end{align}
		in the variable $\Vector{\tau}$ verifies containment in $\dualcircuitSONC$. 
	\label{prop:dualfeasibilty}
\end{proposition}


%

The feasibility problem \cref{prop:dualfeasibilty} forms the basis for our DSONC approach to Lyapunov stability analysis in \cref{sec:StabilityDSONC}.

\section{Polynomial Optimization for Lyapunov Analysis}
\label{sec:PolyLyap}

We now illustrate how polynomial optimization may be utilized to create a sufficient condition for stability. 
For the special case of SOS, this has been described in \cite{ParriloThesis,PapachristodoulouAdvancesSOS,Ahmadi:Parrilo:Stability}, among others.
In order to use polynomial optimization as defined in \cref{def:cpop} to tackle this issue, 
we restrict ourselves to the unconstrained case. 
The following theorem provides a relaxation of \cref{thm:lyapunovstability} using an arbitrary nonnegativity certificate $\struc{\mathbb{GEN}}$ and the domain $\cU = \R^n$.


\begin{theorem}
	\label{thm:gen_stability}
	Let $p_1, p_2 \in \R[\xb]$ be strictly positive on $\R^n \setminus \set{0}$.
	The equilibrium $\xb^\ast = \ob$ of the system \cref{eq:nonlinearsystem} is stable, if there exists a polynomial function $V : \R^n \rightarrow \R$ which satisfies
	\begin{align*}
		V(\ob) \ &= \ 0, \\
		V(\xb) - p_1(\xb) \ &\in \ \mathbb{GEN} \qquad \text{ for all } \xb \in \R^n \setminus \{ \ob \}, \\
		- \dot{V}(\xb) \ &\in \ \mathbb{GEN} \qquad \text{ for all } \xb \in \R^n.
	\end{align*}
	If furthermore
	\begin{align*}
		-\dot{V}(\xb) - p_2(\xb) \ \in \ \mathbb{GEN} \qquad \text{ for all } \xb \in \R^n \setminus \{ \ob \},
	\end{align*}
	then $\xb^\ast$ is asymptotically stable.
\end{theorem}

\begin{proof}
	Let $V$ be polynomial functions satisfying the conditions above.
	In order to prove that this implies (asymptotic) stability, we need to check that the conditions given in \cref{thm:lyapunovstability} are satisfied. 
	Since $V$ is a polynomial function, it is continuously differentiable. 
	The condition $V(\xb) - p_1(\xb) \in \mathbb{GEN}$ for some certificate of nonnegativity $\mathbb{GEN}$ implies that 
	\begin{align*}
		V(\xb) \ \ge \ p_1(\xb) \ > \ 0 \qquad \text{ for all } \xb \in \R^n \setminus \{ \ob \},
	\end{align*}
	and if $- \dot{V}(\xb) \in \mathbb{GEN}$ holds, then $\dot{V}(\xb) \le 0$ for all $\xb \in \R^n$.
	By \cref{thm:lyapunovstability}, stability follows. 
	Further, $-\dot{V}(\xb) - p_2(\xb) \ \in \ \mathbb{GEN}$ implies
	\begin{align*}
		-\dot{V}(\xb) \ \ge \ p_2(\xb) \ > \ 0 \qquad \text{ for all } \xb \in \R^n \setminus \{ \ob \}.
	\end{align*}
	So asymptotic stability follows by \cref{thm:lyapunovstability}.
\end{proof}

In fact, for the case $\mathbb{GEN} = \DSONC$, see \cref{Equation:DSONCCone}, with equilibrium $\xb^\ast = \ob$, \cref{thm:gen_stability} can be simplified.
Polynomials in the DSONC cone have no real zeros outside the origin, with the trivial exception of the zero polynomial, see \cite{Heuer:deWolff:DualityOfSONC}.
Thus, we may omit subtracting positive polynomials, and obtain the following corollary.

\begin{corollary}
	The equilibrium $\xb^\ast = \ob$ of the system \cref{eq:nonlinearsystem} is asymptotically stable, if there exists a polynomial function $V : \R^n \rightarrow \R$ which satisfies that $V$ is not the zero polynomial and
	\begin{align*}
		V(\ob) \ & = \ 0, \\
		V(\xb) \ &\in \ \DSONC \qquad \text{ for all } \xb \in \R^n, \\
		- \dot{V}(\xb) \ &\in \ \DSONC \qquad \text{ for all } \xb \in \R^n.
	\end{align*}
\label{thm:dsonc_stability}
\end{corollary}

\begin{proof}
	This is a direct consequence of the fact that containment in the DSONC cone $\cD$ implies positivity, see \cite[Corollary 4.11]{Heuer:deWolff:DualityOfSONC}, and \cref{thm:gen_stability}.
\end{proof}


Our ultimate goal is to provide a (D)SONC-based procedure for finding Lyapunov functions $V$. 
For now, we assume that we already have a candidate function $V$ at hand and just want to test whether $V$ is a Lyapunov function by using \cref{thm:gen_stability} and \cref{thm:dsonc_stability}. 
We call such a function $V$ a \struc{Lyapunov candidate}.

\begin{example}
	Consider the polynomial system
	\begin{align}
		\begin{split}
			\dot{x_1} \ &= \ -x_1^3 - x_1 - x_1x_3^2, \\ 
			\dot{x_2} \ &= \ x_1^2 - x_2, \\
			\dot{x_3} \ &= \ -x_3,
		\end{split}
	\label{eq:global_stability1}
	\end{align}
	which has a unique equilibrium in the origin.
	We want to use \cref{thm:gen_stability} for different nonnegativity certificates $\mathbb{GEN}$ to show that \cref{eq:global_stability1} is asymptotically stable.
	For this, we consider the Lyapunov candidate $V(\xb) = x_1^2+x_2^2+x_3^2$.
	This function is positive on $\R^n \setminus \set{\ob}$. 
	It is also a sum of monomial squares and can as such trivially be written as a sum of squares of other polynomials.
	It is furthermore trivially contained in the SONC and DSONC cones.

	The negative derivative of this Lyapunov candidate is
	\begin{align*}
		-\dot{V}(\xb) \ &= \ - \left(\frac{\partial V}{\partial \xb}\right)\T \dot{\xb} \\
		\ &= \ 2x_1^4 + 2x_1^2x_3^2 - 2x_1^2x_2 + 2x_1^2 + 2x_2^2 + 2 x_3^2.
	\end{align*}
	If we can show that $-\dot{V}(\xb) \in \mathbb{GEN}$ for any of the following different choices of $\mathbb{GEN}$, then the equilibrium of \cref{eq:global_stability1} is (asymptotically) stable.
	\begin{enumerate}[topsep=5pt]
		\item First, consider the case where $\mathbb{GEN}$ denotes the SOS cone.
		We can easily verify that $-\dot{V}$ can be written as
		\begin{align*}
			-\dot{V}(\xb) \ = \ &\left(-\frac{1}{2}\sqrt{6}x_1^2 + \frac{1}{2}\sqrt{6}x_2 \right)^2 + \left(\sqrt{2}x_3 \right)^2 + \left(\sqrt{2}x_1 \right)^2 \\
			&+ \left(\sqrt{2}x_1x_3 \right)^2 + \left(\frac{1}{2}\sqrt{2}x_1^2 + \frac{1}{2}\sqrt{2}x_2 \right)^2,
		\end{align*}
		and is thus a sum of squared polynomials.
		By \cref{thm:gen_stability}, \cref{eq:global_stability1} is stable in its equilibrium $\xb^\ast = \ob$.

		In fact, we can even show \emph{asymptotic} stability.
		For this consider $P(\xb) = -\dot{V}(\xb) - p_2(\xb)$, where $p_2$ can be chosen as, e.g., $p_2(\xb) =  x_1^2 + x_2^2$.
		It can be shown that $P$ is still SOS, and thus asymptotic stability follows.
		\item Let now $\mathbb{GEN}$ be the SONC cone.
		To verify that $-\dot{V}$ can be written as a sum of nonnegative circuit polynomials, consider the representation
		\begin{align*}
			- \dot{V}(\xb) \ = \ \underbrace{\left(2x_2^2 - 2x_1^2x_2 + 2x_1^4 \right)}_{= c(x_1, x_2)} + 2x_1^2x_3^2 + 2x_1^2  + 2 x_3^2.
		\end{align*}
		If we can show that $c(x_1, x_2)$ is a nonnegative circuit polynomial, then the stability of the system \cref{eq:global_stability1} follows.
		Note first that for the support $\set{\Matrix{4\\0}, \Matrix{0\\2}, \Matrix{2\\1}}$ of $c$ we have 
		\begin{align*}
			\Matrix{2\\1} \ = \ \frac{1}{2} \Matrix{4\\0} + \frac{1}{2} \Matrix{0\\2}.
		\end{align*}
		This shows that $c$ is supported on a circuit whose outer points $\Matrix{4\\0}$ and $\Matrix{0\\2}$ are even.
		If we now compute the circuit number as introduced in \cref{thm:nonnegativecircuit}, then we get 
		\begin{align*}
			\Theta_c \ = \ (4 \cdot 4)^{\frac{1}{2}} \ = \ 4 \ > \ |-2|,
		\end{align*}
		which proves nonnegativity.
		To see that the system is \emph{asymptotically} stable, we can e.g. subtract $p_2(\xb) = x_1^2  + x_3^2$ from $-\dot{V}$.
		\item Finally, we choose $\mathbb{GEN} = \DSONC$. 
		The positive support of $-\dot{V}(\xb)$ is given by 
		\begin{align*}
			A^+ \ = \ \set{\Matrix{4\\0\\0}, \Matrix{2\\0\\2}, \Matrix{2\\0\\0}, \Matrix{0\\2\\0}, \Matrix{0\\0\\2}}.
		\end{align*}
		Since we have only one interior point $\betab = \Matrix{2\\1\\0} \in A^- = \set{\betab}$, we only need to verify containment in the cone $\circuitDSONC$.
		Note that instead of verifying the nonnegativity of $-\dot{V}(\xb)$, we may equivalently consider $- \frac{1}{2} \dot{V}(\xb)$.
		This has the advantage that all involved coefficients now have absolute value $1$, which, as we see in what follows, simplifies our computations.
		We thus attempt to find a solution to the feasibility problem \cref{prop:dualfeasibilty} for $- \frac{1}{2} \dot{V}(\xb)$.
		Observe that $\ln\left(\frac{|c_{\betab}|}{c_{\alpb}}\right) = \ln(1) = 0$ for all $\alpb \in A^+$, where $c_{\alpb}, c_{\betab}$ stand for the coefficients of $- \frac{1}{2} \dot{V}(\xb)$.
		Thus, the feasibility problem admits the trivial solution $\Vector{\tau} = \ob$.
		It follows by \cref{thm:dsonc_stability} that \cref{eq:global_stability1} is asymptotically stable.
	\end{enumerate}
\label{ex:global_stability1}
\end{example}

To illustrate the difference between stability and asymptotic stability, we examine another small example.

\begin{example}[Linearized hanging pendulum]
	Consider the system
	\begin{align*}
		\dot{x_1} &\ = \ x_2, \qquad 
		\dot{x_2} \ = \ -\sin(x_1),
	\end{align*}
	which can be interpreted as a simple normalized hanging pendulum, where $x_1$ is the angle of deflection of the pendulum with respect to the vertical axis.
	The equilibrium $(x_1, x_2) = (0, 0)$ of this system describes the state in which the pendulum is hanging straight down.
	For small angles $x_1$, we can linearize this system using Taylor expansion and instead consider
	\begin{align*}
		\dot{x_1} &\ = \ x_2, \qquad 
		\dot{x_2} \ = \ -x_1.
	\end{align*}
	We can, for example, choose $V(x_1, x_2) = \frac{1}{2}x_1^2 + \frac{1}{2}x_2^2$ as a Lyapunov function for this system.
	With this choice we get that $-\dot{V}$ is the zero polynomial, so we are able to certify stability, but not asymptotic stability.

	Indeed, the pendulum moves along a trajectory corresponding to constant energy when perturbed, as illustrated in the phase plot in \cref{fig:pendulum}. 
	That is, once the pendulum is set in motion it will continue to swing on the same trajectory (since we assume no outside influences).
\end{example}

\begin{figure}[hbt]
	\includegraphics[draft=false, width=.75\textwidth]{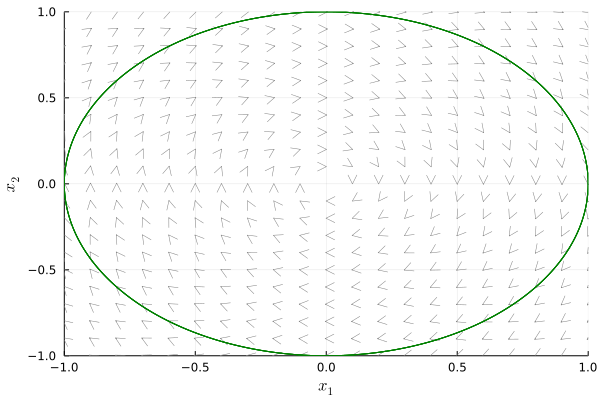}
	\caption{Phase plot for the linearized hanging pendulum}
	\label{fig:pendulum}
\end{figure}

As we already pointed out, the current approach to Lyapunov search via polynomial optimization is using SOS certificates (and semidefinite programming).
In fact, one can show that there exists an SOS Lyapunov function whenever a polynomial Lyapunov function exists. 

\begin{lemma}[{\cite[Lemma 5.1]{Ahmadi:Parrilo:Stability}}]
	If a polynomial system of the form \cref{eq:nonlinearsystem} has a polynomial Lyapunov function, then it also has a SOS Lyapunov function.
\end{lemma}

To see that this lemma holds, one can just replace the polynomial Lyapunov function $V$ by its square.
It is, however, not clear that the negative derivative of this squared Lyapunov function is also a sum of squares for general polynomial systems.
For this to hold, we need to additionally require the system to be defined by \emph{homogeneous} polynomials, see \cite[Theorem 5.3]{Ahmadi:Parrilo:Stability}.


In the rest of the paper we investigate how SONC and DSONC nonnegativity certificates can be used to find Lyapunov functions for given polynomial systems as an alternative to the SOS approach.

\subsection{Global Stability Analysis via the SONC Cone}
\label{sec:StabilitySONC}

When searching for polynomial Lyapunov functions, the involved coefficients are no longer fixed, but additional variables in our optimization problem. 
Thus, we need to investigate how variable coefficients affect our computations before we apply polynomial optimization techniques to stability analysis.

\begin{proposition}
	Verifying containment of a polynomial with support $A = A^+ \cup A^-$ and \emph{variable} coefficients $\{c_{\alpb}\}_{\alpb \in A^+} \subset \R_{\ge 0}$ and $\{c_{\betab}\}_{\betab \in A^-} \subset \R$ in the SONC cone $\signedSONC$ can be written as an REP.
	\label{cor:primalREP}
\end{proposition}

\begin{proof}
	Let $A \subset \N^n$ be a support set with fixed sign distribution $A = A^+ \cup A^-$, and let $p$ be a polynomial supported on $A$ with variable coefficients $\{c_{\alpb}\}_{\alpb \in A^+}$ and $\{c_{\betab}\}_{\betab \in A^-}$. 
	Then it is a direct consequence of \cref{thm:sagecontainment} that $p \in \signedSONC$ if and only if the coefficients take values for which the relative entropy program 
	\begin{align}
		\begin{aligned}
			\min \qquad & 0 \\
			\text{subject to } \qquad & \sum_{\betab \in A^-} c_{\alpb}^{(\betab)} &\ = &\ c_{\alpb} & \qquad \text{ for all } \ \alpb \in A^+ , & \\
			& D(\Vector{v}^{(\betab)}, e \cdot \Vector{c}^{(\betab)}) &\ \le &\ c_{\betab} & \qquad\text{ for all } \ \betab \in A^-, &\text{ and } \\
			& \sum_{\alpb \in A^+} v_{\alpb}^{(\betab)} (\alpb - \betab) &\ = &\ 0 & \qquad \text{ for all } \ \betab \in A^-, &
		\end{aligned}
		\label{eq:varREP}
	\end{align}
	in the additional variables $\Vector{c}^{(\betab)}, \Vector{v}^{(\betab)} \in \R^{A^+}_{\ge 0}$, $\betab \in A^-$, has a feasible solution.
\end{proof}

\begin{proposition}
	\label{prop:LyapSONCREP}
	Searching for a Lyapunov function in the SONC cone which verifies that the equilibrium $\xb^\ast = \ob$ of the system $\Vector{f} \in \R[\xb]^n$ of the form \cref{eq:nonlinearsystem} is stable can be done via relative entropy programming.
\end{proposition}

\begin{proof}
	We have seen in \cref{cor:primalREP} that verifying containment of a polynomial with variable coefficients in the SONC cone can be done by solving the REP \cref{eq:varREP}.
	If we now use this approach to search for a Lyapunov function, then we have to verify containment of both the Lyapunov function and its negative derivative in the SONC cone.
	Let $A \subset \N^n$ be a suitable support set with fixed sign distribution $A = A^+ \cup A^-$, and let $V$ be a polynomial Lyapunov candidate supported on $A$ with variable coefficients  $\{c_{\alpb}\}_{\alpb \in A^+}$ and $\{c_{\betab}\}_{\betab \in A^-}$, and let $p \in \R^A$ be a strictly positive polynomial with fixed coefficients $\set{\epsilon_{\alpb}}_{\alpb \in A}$. 
	We apply \cref{thm:gen_stability} for $\mathbb{GEN} = \ASONC$, which first requires that $V - p \in \ASONC$.
	I.e., we need to solve the relative entropy feasibility problem 
	\begin{align}
		\begin{aligned}
			\min \qquad & 0 \\
			\text{subject to } \qquad & \sum_{\betab \in A^-} c_{\alpb}^{(\betab)} &\ = &\ c_{\alpb} - \epsilon_{\alpb} & \qquad \text{ for all } \ \alpb \in A^+ , & \\
			& D(\Vector{v}^{(\betab)}, e \cdot \Vector{c}^{(\betab)}) &\ \le &\ c_{\betab} - \epsilon_{\betab} & \qquad\text{ for all } \ \betab \in A^-, &\text{ and } \\
			& \sum_{\alpb \in A^+} v_{\alpb}^{(\betab)} (\alpb - \betab) &\ = &\ 0 & \qquad \text{ for all } \ \betab \in A^-, &
		\end{aligned}
	\end{align}
	given by \cref{cor:primalREP}.
	Next, $-\dot{V}$ has to be contained in the SONC cone. 
	For this, note that the coefficients of $-\dot{V}$, denoted by $\Vector{d}(\Vector{c})$, are given by the variable coefficients $\Vector{c}$ of $V$ and the system $\Vector{f}$ as is \cref{eq:nonlinearsystem} through linear equations. 
	This translates to adding linear constraints to the above optimization problem. 
	Let $B = B^+ \cup B^-$ denote the support set of $- \dot{V}$ together with some fixed sign distribution.
	We add another set of constraints of type \cref{eq:varREP} in the additional variables $\Vector{d}^{(\betab)}, \Vector{w}^{(\betab)} \in \R^{B^+}_{\ge 0}$ for $\betab \in B^-$:
	\begin{align*}
		\begin{aligned}
			\sum_{\betab \in B^-} d_{\alpb}^{(\betab)} &\ = &\ d_{\alpb}(\Vector{c}) & \qquad \text{ for all } \ \alpb \in B^+ , & \\
			D(\Vector{w}^{(\betab)}, e \cdot \Vector{d}^{(\betab)}) &\ \le &\ d_{\betab}(\Vector{c}) & \qquad \text{ for all } \ \betab \in B^-, &\text{ and } \\
			\sum_{\alpb \in B^+} w_{\alpb}^{(\betab)} (\alpb - \betab) &\ = &\ 0 & \qquad \text{ for all } \ \betab \in B^-. &
		\end{aligned}
	\end{align*}
	This ensures that $-\dot{V} \in \SONC_{B}$, so for any solution of our optimization problem the stability of the equilibrium of $\Vector{f}$ follows by \cref{thm:gen_stability}.
	Since all constraints are either (relative) entropy or linear constraints, the constructed problem is an REP.
\end{proof}

We now develop an algorithm for computing Lyapunov functions via the SONC cone.
The goal is finding coefficients $\Vector{c} \in \polys$ for some given support set $A \subset \N^n$, such that $V(\xb) = \sum_{\alpb \in A} c_{\alpb} \xb^{\alpb}$ satisfies the conditions in \cref{thm:gen_stability}.
We want to use \cref{thm:sagecontainment} and \cref{thm:gen_stability}.
Thus, the algorithm has to compute the polynomials
\begin{align*}
	P(\xb) = V(\xb) - p_1(\xb)
\end{align*}
with variable coefficients $\Vector{c} \in \polys$, and
\begin{align*}
	P^\prime(\xb) = -\dot{V}(\xb) - p_2(\xb)
\end{align*}
with coefficients $\Vector{d}(\Vector{c}) \in \polys$, which linearly depend on $\Vector{c}$. 
Depending on whether we want to certify stability or asymptotic stability, $p_2$ can be chosen to be the zero polynomial.
We assume that $p_1$ and $p_2$ are polynomials, which, in addition to being positive on $\R^n \setminus \set{\ob}$, have support sets that are contained in $A$.
This is not required by \cref{thm:gen_stability}, but is has the advantage that we avoid adding additional negative terms.
To determine suitable values for the coefficients $\Vector{c}$, we use the optimization problem developed in the proof of \cref{prop:LyapSONCREP}. 
This method involves choosing a fixed sign distribution, and we want to choose a distribution that is likely to lead to a feasible solution.
To this end, we make an informed choice for the signs of the coefficients as follows:
\begin{enumerate}
	\item Compute the subsets of the supports of $P$ and $P^\prime$ which are contained in $(2\N)^n$.
	\item Compute the Newton polytopes of $P$ and $P^\prime$. 
\end{enumerate}

For any polynomial to be nonnegative, the exponents that are vertices of the Newton polytope have to be even and the corresponding coefficients need to be positive. 
This motivates the following assumptions.
\begin{enumerate}[start=3]
	\item Constrain all those variable coefficients to be zero that correspond to \emph{odd} vertices of the computed Newton polytopes.
	\item Constrain all coefficients corresponding to vertices of the Newton polytope to be nonnegative.
\end{enumerate}
There is no obvious best choice to distribute signs for the remaining variable coefficients.
But since the goal is constructing nonnegative polynomial, we aim at creating as many positive terms as possible. 
We do this by adding the following constraints.
\begin{enumerate}[start=5]
	\item Coefficients that have not been assigned a sign yet are constrained to be nonnegative if the corresponding exponent is even.
\end{enumerate}
Since the variable coefficients depend on each other, the number of redundant constraints can become very large.
Thus, our final step is removing all redundant constraints that have been added so far.
We do this by solving the following linear program. 
\begin{enumerate}[start=6]
	\item Assume that our constraints are given as a list of the form $[c_i \le 0 \ : \ i \in [l]]$, where the $c_i$ are linear functions. 
	Further assume that these constraints are already known to be not redundant for $i \in [k]$ for some $k < l$ and initialize the list containing all non-redundant constraints as $\mathtt{constraints} = [c_i \le 0 \ : \ i \in [k]]$. 
	Then for all remaining linear functions $c_j$ with $j \in \set{k+1, \ldots, l}$ we solve 
	\begin{align*}
		\max & \quad c_j \\
		\text{subject to } & \quad \mathtt{constraints} 
	\end{align*}
	If this LP has a positive feasible solution, then the constraint defined by $c_j$ is not redundant and we add it to the list $\mathtt{constraints}$ before repeating the procedure for $c_{j+1}$. 
	If the LP does not have a positive feasible solution, then we can disregard the constraint given by $c_j$.
\end{enumerate}

We refer to the method of choosing signs given by (1)-(6) as $\struc{\mathtt{distribute\_signs}}$. 		
The corresponding pseudocode can be found in the Appendix, see \cref{Alg:distributeSigns}
Note that step (4), while being a simple and reasonable assumption, is not necessarily the best possible choice.
We illustrate this fact in \cref{ex:5}.
In general, the sign choice for the coefficients which, if possible, leads to a Lyapunov function is not unique.
We summarize the described method of searching Lyapunov functions via the SONC cone in \cref{Alg:primalREPalg}.

\begin{algorithm}~

	\noindent\INPUT{$A \subset \N^n \setminus \set{\ob}$, $\Vector{f} \in \R[\xb]^n, $ $p_1, p_2 \in \ASONC$ strictly positive on $\R^n \setminus \set{0}$} \\
	\OUTPUT{$V \in \ASONC$ strictly positive on $\R^n \setminus \set{0}$}
	\begin{algorithmic}[1]
		\State $V \gets \mathtt{poly}(\Vector{c}, A)$ 
		\COMMENT{\parbox[t]{.4\linewidth}{polynomial with variable coefficient vector $\Vector{c}$ and support $A$}}
		\State compute $-V' \ = \ \mathtt{poly}(\Vector{d}(\Vector{c}), A^\prime)$ 
		\COMMENT{\parbox[t]{.4\linewidth}{negative derivative of $V$; coefficients depend on $\Vector{c}$}}
		\State $P \gets V - p_1$, $P' \gets -V' - p_2$
		\State add sign constraints for $\Vector{c}, \Vector{d}(\Vector{c})$ according to \texttt{distribute\_signs}
		
			\State add constraints $P \in \signedSONC$ and $P' \in \SONC_{(A^\prime)^+, (A^\prime)^-}$


		\State solve feasibility problem

		\State \Return{V}
	\end{algorithmic}
\label{Alg:primalREPalg}
\end{algorithm}

Note that Lyapunov functions found via this algorithm certify the asymptotic stability of the underlying system.
If only a certificate for stability is required, then we can choose $p_2$ as the zero polynomial.

\subsection{Global Stability Analysis via the DSONC Cone}
\label{sec:StabilityDSONC}

As discussed in \cref{sec:DualSONC}, containment in the DSONC cone is a certificate of nonnegativity.
Hence, setting $\mathbb{GEN} = \signedDSONC$ for some support set $A = A^+ \cup A^-$ yields a sufficient criterion for the asymptotic stability of an equilibrium of \cref{eq:nonlinearsystem}.
It further holds that optimizing over the DSONC cone can be done via linear programming, see \cref{prop:dualfeasibilty}. 
Being able to find Lyapunov functions using linear programming would be highly desirable, so we first investigate how the optimization problem using the DSONC cone to certify nonnegativity changes if the involved coefficients are variable instead of fixed. 

\begin{proposition}
	Verifying containment of a polynomial with support $A = A^+ \cup A^-$ and \emph{variable} coefficients $\{c_{\alpb}\}_{\alpb \in A^+}$ and $\{c_{\betab}\}_{\betab \in A^-}$ in the cone $\signedDSONC = \sum_{\betab \in A^-} \circuitDSONC$ can be written as a linear program.
\end{proposition}

\begin{proof}
	First, consider the case $A^- = \set{\betab}$.
	To verify containment of the coefficient vector $\Vector{c}$ in the dual SONC cone $\dualcircuitSONC$, we need to find some $\Vector{\tau} \in \R^n$ such that 
	\begin{align}
		\ln \left(\frac{\abs{c_{\betab}}}{c_{\alpb}}\right) \ \le \ \left(\alpb - \betab\right)\T \Vector{\tau}
		\label{eq:dualcontainment}
	\end{align}
	for all $\alpb \in A^+$.
	To write this as a system of \emph{linear} inequalities for variable coefficients, we define variables $\rho^{(\alpb)} = \ln(c_{\alpb})$ and $\sigma^{(\betab)} = - \ln(|c_{\betab}|)$. 
	With this \cref{eq:dualcontainment} becomes
	\begin{align}
		0 \ \le \ \left(\alpb - \betab\right)\T \Vector{\tau} + \sigma^{(\betab)} + \rho^{(\alpb)},
	\label{eq:varconst_linear}
	\end{align}
	which is linear.
	For the case of general support sets, we use the fact that, by construction, any DSONC polynomial $p(\xb) = \sum_{\alpb \in A^+} c_{\alpb} \xb^{\alpb} + \sum_{\betab \in A^-} c_{\betab} \xb^{\betab}$ can be written as 
	\begin{align*}
		p(\xb) \ = \ \sum_{\betab \in A^-} p_{\betab}(\xb),
	\end{align*}
	where $p_{\betab} = \sum_{\alpb \in A^+} c^{(\betab)}_{\alpb} \xb^{\alpb} + c_{\betab} \xb^{\betab} \in \circuitDSONC$.
	Thus, we only get additional linear constraints of the form $\sum_{\betab \in A^-} c^{(\betab)}_{\alpb} = c_{\alpb}$ and the claim follows.

\end{proof}

Unfortunately, this is no longer the case if we want to solve the optimization problem from \cref{thm:gen_stability} which verifies asymptotic stability of a polynomial system. 
The reason for this is that we now need to solve two feasibility problems which depend on each other.
Specifically, the coefficients of the derivative of the Lyapunov candidate are dependent on the variable coefficients of the Lyapunov candidate $V(\xb)$.
%
%
%
 We can, however, formulate the problem as an REP, as the following proposition shows.

 \begin{proposition}
	\label{prop:LyapDSONCREP}
	 The problem in \cref{thm:gen_stability} for $\mathbb{GEN} = \signedDSONC$ can be stated as an REP.
 \end{proposition}

 \begin{proof}
	We first note that for any fixed $\betab \in A^-$ we can replace the condition 
	\begin{align*}
		\ln\left(\frac{|c_{\betab}|}{c_{\alpb}}\right) \ \le \ (\alpb - \betab)\T \Vector{\tau} \qquad \text{ for  all } \ \alpb \in A^+
	\end{align*}
	for containment in $\circuitDSONC$ by 
	\begin{align}
		|c_{\betab}| \cdot \ln\left(\frac{|c_{\betab}|}{c_{\alpb}}\right) \ \le \ (\alpb - \betab)\T \Vector{\tilde\tau} \qquad \text{ for  all } \ \alpb \in A^+,
		\label{eq:REPcondition}
	\end{align}
	where we set $\Vector{\tilde\tau} = |c_{\betab}|\Vector{\tau}$.
	We then observe that the coefficients of the derivative $\dot V(\xb)$ depend \emph{linearly} on the coefficients of the Lyapunov candidate $V(\xb)$.
	Thus, all needed constraints are either relative entropy constraints of the form \cref{eq:REPcondition}, or linear constraints.  
 \end{proof}

The reason why we cannot formulate this as a linear problem as we did in \cref{eq:varconst_linear} is that we can no longer simply substitute $\rho^{(\alpb)} = \ln(c_{\alpb})$ and $\sigma^{(\betab)} = - \ln(|c_{\betab}|)$ and recover the true values for the variables $c_{\alpb}$ and $c_{\betab}$ after solving the optimization problem in the variables $\rho^{(\alpb)}$ and $\sigma^{(\betab)}$. 
This is due to the fact that now all variables $c_{\alpb}$ and $c_{\betab}$ describing coefficients of $-\dot{V}$ now in turn linearly depend on variable coefficients of $V$.
This means that while we can still use the DSONC cone to find Lyapunov functions, it no longer provides a computational advantage over the primal SONC cone. 
Since finding Lyapunov functions over either cone is an REP and the DSONC cone is a subset of the SONC cone, using the larger SONC cone is always preferable.

\section{Examples}
\label{sec:Examples}

In this section we explore some selected examples of polynomial systems whose stability we try to verify using the previously introduced algorithms, as well as the SOS approach.
The algorithms presented in this paper were implemented in \texttt{Python} and all Lyapunov functions found via these methods were verified using the POEM software \cite{poem:software}.
For the SOS results we have used the Julia package \texttt{SumOfSquares.jl} \cite{SumOfSquares}.
All computations were executed on an Apple M2 with 24GB RAM.
The full code and all examples can be accessed via \\

\noindent
\url{https://moto.math.nat.tu-bs.de/appliedalgebra_public/sonclyapunov_initial}. \\



In all the examples in this section the positive polynomials $p_1$ and $p_2$ as required by \cref{thm:gen_stability} for the SONC and SOS computations are chosen as variable polynomials of the form 
\begin{align*}
	p(\xb) \ = \ \sum_{\alpb \in A^+} \eps_{\alpb} \xb^{\alpb},
\end{align*}
where $A^+$ denotes the set of even exponents (of $V$ or $-\dot{V}$) with positive coefficients as determined in the procedure \texttt{distribute\_signs}, and the $\eps_{\alp}$ are variables which we constrain such that they are nonnegative and sum up to at least $10^{-3}$.

\begin{example}[{\cite[Example 1]{She:etal:PolyLyap}}]
	We begin with a simple example in $2$ variables in which all tested algorithms yield comparable results.
	The system we investigate is
	\begin{align}
		\begin{split}
			\dot{x_1} & \ = \ -x_1 - \frac{3}{2}x_1x_2^3,\\ 
			\dot{x_2} & \ = \ -x_2^3 + \frac{1}{2}x_1^2x_2^2.
		\end{split}
		\label{ex:levelsets}
	\end{align}
	This system has a unique equilibrium in the origin, which means that all the presented algorithms are applicable. 
	We attempt to search for a Lyapunov function of the form $V(\xb) = c_1 x_1^2 + c_2 x_2^2$, and the \cref{Alg:primalREPalg} find
	\begin{align*}
		V(\xb) & \ = \ \frac{3}{4}x_1^2 + \frac{1}{4}x_2^2.
	\end{align*}
	The SOS approach finds $V(\xb) = \frac{2}{3}x_1^2 + 2x_2^2$.
	The phase plot of \cref{ex:levelsets} and the level sets of this $V(\xb)$ are shown in \cref{fig:levelsets}.
	\label{ex:1}
\end{example}

\begin{figure}[hbt]
	\includegraphics[draft=false, width=.75\textwidth]{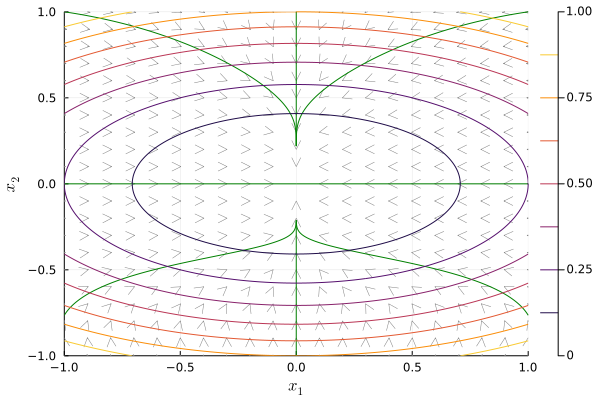}
	\caption{Phase plot of the system \cref{ex:levelsets} including level sets of the found Lyapunov function.}
	\label{fig:levelsets}
\end{figure}

\begin{example}[{\cite[Example 1]{PapachristodoulouPrajna:LyapunovSOS}}]
	We now attempt to verify the stability of the system
	\begin{align*}
		\begin{array}{l l}
			\dot{x_1} \ = \ -x_1^3 + 4x_2^3 - 6x_3x_4, \qquad & \dot{x_4} \ = \ x_1x_3 + x_3x_6 - x_4^3, \\
			\dot{x_2} \ = \ -x_1 - x_2 + x_5^3, \qquad & \dot{x_5} \ = \ -2x_2^3 - x_5 + x_6,\\
			\dot{x_3} \ = \ x_1x_4 - x_3 + x_4x_6, \qquad & \dot{x_6} \ = \ -3x_3x_4 - x_5^3 - x_6
		\end{array}
	\end{align*}

	in $6$ variables of degree $3$ in the unique equilibrium $\xb^\ast = \ob$ by searching for a Lyapunov function of the form $V(\xb) = \sum_{j = 1}^6 c_j x_j^2$.
	However, \cref{Alg:primalREPalg} does not find a suitable Lyapunov function of this form.
	The SOS approach is also unable to find a Lyapunov function of degree $2$.

	For the SONC based approach we find that it is sufficient to add terms $c_7 x_2^4$ and $c_8 x_5^4$ to the above $V(\xb)$. 
	Now we find the Lyapunov function
	\begin{align*}
		V(\xb) & \ = \ 
		3.426139247207665x_6^2 + 1.7130696236038325x_5^4 + 10.278417741622995x_4^2 \\
		& \qquad + 3.426139247207665x_2^4 + 1.7130696236038325x_1^2.
	\end{align*}

	The Lyapunov function of degree $4$ found via SOS is 
	\begin{align*}
		V(\xb) \ = \
		& 1.7448099617085968x_{2}^{4} + 0.8723988328198938x_{5}^{4} \\
		&+ 0.0010865978195663756x_{1}x_{3}x_{4} + 0.001522455857803553x_{3}x_{4}x_{5} \\
		&+ 0.9336697311290476x_{1}^{2} + 0.41698555619326955x_{1}x_{2} \\
		&+ 0.038922626081588314x_{1}x_{5} + 0.20849276323163046x_{2}^{2} \\ 
		&+ 0.03892256726985556x_{2}x_{5} + 2.6975458027715975x_{3}^{2} + 1.9302454720518627x_{4}^{2} \\
		&+ 0.38197987972890357x_{5}^{2} + 1.6669816282831071x_{6}^{2}.
	\end{align*}
	\label{ex:2}
\end{example}



\begin{example}
	We now investigate the system
	\begin{align*}
		\dot{x_1} & \ = \ -x_1^3x_2^2 + x_1x_2^4 - x_1x_3^2, \\
		\dot{x_2} & \ = \ -2x_1^2x_2^3 + 5x_1^2x_2x_3^2 + x_2^2x_3^3, \\
		\dot{x_3} & \ = \ -3x_1^2x_2^2x_3 - x_2^3x_3^2 - x_3^5 + x_1^2x_3.
	\end{align*}
	If we try to compute a degree $2$ Lyapunov function using the algorithms as described in this paper, then no result is returned.
	However, we can find the Lyapunov function 
	\begin{align*}
		V(\xb) = \frac{1}{2}x_1^2 + \frac{1}{2}x_2^2 + \frac{1}{2}x_3^2
	\end{align*}
	by making the following modification to the SONC approach.
	Recall that in the allocation of the signs in \texttt{distribute\_signs} we make the choice that all terms with even exponents have to be positive.
	In this case, we omit this restriction for all terms in the interior of the Newton polytope and instead make these terms negative.

	It turns out that the SOS method does not find a degree $2$ Lyapunov function for this system, even though the above function $V$ is clearly a sum of squares.
	The reason is that the negative derivative of $V$ is given by 
	\begin{align*}
		-\dot{V}(\xb) = x_3^6 - 2x_1^2x_2^2x_3^2 + x_1^2x_2^4 + x_1^4x_2^2.
	\end{align*}
	This is the homogenized version of the Motzkin polynomial 
	which is famously not SOS, see \cite{Motzkin:AMGMIneq}.
	We attempted to increase the degree up to $8$, but were unable to get a verifiable result using the SOS approach. 
	That is, either there was no result returned, or (e.g. for the cases of degrees $4, 6$ and $8$) results were returned, but the problem status returned by the solver was ``unknown'' each time. 
	This means that the returned function possibly violates the set constraints. 
	In these cases, we checked the nonnegativity of the returned functions individually and all of them turned out to be unbounded and thus not viable Lyapunov functions.
	\label{ex:5}
\end{example}

So far, we have only presented examples where we are able to find Lyapunov functions with the circuit-based algorithms.
However, we cannot guarantee success for arbitrary polynomial systems, as the following example demonstrates.

\begin{example}[{\cite[Example 2]{She:etal:PolyLyap}}]
	Our algorithms are unsuccessful for Lyapunov candidates of degree up to $8$ for the system
	\begin{align*}
		\dot{x_1} & \ = \ x_2-x_1^3 + x_1x_2^4 \\
		\dot{x_2} & \ = \ -x_1^3 - x_2^5.
	\end{align*}
	By searching for SOS Lyapunov functions, we find 
	\begin{align*}
		V(\xb) & \ = \ 0.044373193847826196x_{1}^{5}x_{2} + 0.22554047027447183x_{1}^{3}x_{2}^{3} \\
		&+ 0.17304563146952115x_{1}^{2}x_{2}^{4} + 0.4653621747641972x_{1}x_{2}^{5} + 0.757506544645783x_{2}^{6} \\
		&+ 0.5557146174641293x_{1}^{4} + 1.1114262218679964x_{2}^{2},
	\end{align*}
	which is indeed not a SONC polynomial.
	\label{ex:6}
\end{example}

\section{Limitations and Outlook}
\label{sec:Limitations}

\subsection{Existence of Polynomial Lyapunov functions}

Even though we only study the stability of \emph{polynomial} dynamical systems \cref{eq:nonlinearsystem}, it can happen that there exists no polynomial Lyapunov function despite the system being globally asymptotically stable.
An example of this is given in the following theorem from \cite{Ahmadi:Krstic:Parrilo:stablenotpolynomial}.

\begin{theorem}[\cite{Ahmadi:Krstic:Parrilo:stablenotpolynomial}]
	The polynomial vector field 
	\begin{align*}
		\dot{x_1} & \ = \ -x_1 + x_1x_2, \\
		\dot{x_2} & \ = \ -x_2
	\end{align*}
	is globally asymptotically stable in the equilibrium $\xb^\ast = \ob$, but does not admit a polynomial Lyapunov function.
\end{theorem}

In this case, \cref{Alg:primalREPalg} is naturally unsuitable for proving stability.



\subsection{Existence of SONC-Lyapunov Functions}

We further point out once more that even if a polynomial Lyapunov function exists, it is not clear whether we can also find a SONC-Lyapunov function.
One indicator for the existence of a SONC-Lyapunov function is the support structure of the polynomials $\dot{x_i} = f_i$ defining the dynamical system.
If, for example, the $f_i$ can be decomposed into suitable ``negative circuit polynomials'', i.e., polynomials which satisfy that
\begin{enumerate}
	\item the support sets are minimally affine dependent, and
	\item if $\alpb$ is a vertex of the Newton polytope, then $\alpb + \Vector{e}_i \in (2 \N)^n$ and $c_{\alpb} < 0$, where $\Vector{e}_i$ denotes the $i$-th unit vector and $c_{\alpb}$ denotes the coefficient corresponding to $\alpb$,
\end{enumerate}
then \cref{Alg:primalREPalg} will find a Lyapunov function whenever the inner terms of these negative circuit polynomials are sufficiently small. 
Here, ``sufficiently small'' means that the coefficients of these terms have absolute values which are bounded by a circuit-number-like condition.
If this is satisfied, then we can even find a Lyapunov function of the form $V(\xb) = \sum_{j = 1}^n c_j x_j^2$.

\subsection{Extension to Constrained Systems}

So far, we have only considered the unconstrained case.
However, in numerous applications we have to consider the Lyapunov search problem on restrictive constraint sets.
While we have not yet implemented the constrained case for the (D)SONC approaches presented in this article, we are planning to extend the algorithms in future work.
A natural next step in this direction is using the \emph{X-SAGE} methods, see \cite{Murray:Chandrasekaran:Wierman:SigOptREP, Murray:Naumann:Theobald:Sublinear,Dressler:Murray:AlgebraicPerspectives}.
This will also enable us to explore more real-world examples and applications.

\subsection{Optimal Sign Distributions}

As described in \cref{sec:StabilitySONC} and illustrated again in \cref{ex:5}, our method $\mathtt{distribute\_signs}$ is not yet optimal. 
An interesting line of future research would be to figure out how to, either, make the best possible choice for the sign distribution, or circumvent the necessity of having to choose signs altogether.


\bibliographystyle{amsalpha}
\bibliography{lyapunov_polyopt}

\newcommand{\etalchar}[1]{$^{#1}$}
\providecommand{\bysame}{\leavevmode\hbox to3em{\hrulefill}\thinspace}
\providecommand{\MR}{\relax\ifhmode\unskip\space\fi MR }
\providecommand{\MRhref}[2]{%
  \href{http://www.ams.org/mathscinet-getitem?mr=#1}{#2}
}
\providecommand{\href}[2]{#2}
\begin{thebibliography}{DHNdW20}

\bibitem[AKP11]{Ahmadi:Krstic:Parrilo:stablenotpolynomial}
A.A. Ahmadi, M.~Krstic, and P.A. Parrilo, \emph{A globally asymptotically
  stable polynomial vector field with no polynomial {Lyapunov} function}, 2011
  50th {IEEE} {Conference} on {Decision} and {Control} and {European} {Control}
  {Conference}, {CDC}-{ECC} 2011, 2011, pp.~7579--7580 (English (US)).

\bibitem[AP13]{Ahmadi:Parrilo:Stability}
A.A. Ahmadi and P.A. Parrilo, \emph{Stability of polynomial differential
  equations: Complexity and converse lyapunov questions}, 2013,
  arXiv:1308.6833.

\bibitem[BPT13]{Blekherman:Parrilo:Thomas}
G.~Blekherman, P.A. Parrilo, and R.R. Thomas, \emph{Semidefinite optimization
  and convex algebraic geometry}, MOS-SIAM Series on Optimization, vol.~13,
  SIAM and the Mathematical Optimization Society, Philadelphia, 2013.

\bibitem[CS16]{Chandrasekaran:Shah:SAGE-REP}
V.~Chandrasekaran and P.~Shah, \emph{Relative entropy relaxations for signomial
  optimization}, SIAM Journal on Optimization \textbf{26} (2016), no.~2,
  1147--1173.

\bibitem[DHNdW20]{Dressler:Heuer:Naumann:deWolff:DualSONCLP}
M.~Dressler, J.~Heuer, H.~Naumann, and T.~de~Wolff, \emph{{Global Optimization
  via the Dual SONC Cone and Linear Programming}}, Proceedings of the 45th
  International Symposium on Symbolic and Algebraic Computation (New York, NY,
  USA), ISSAC '20, Association for Computing Machinery, 2020, pp.~138--145.

\bibitem[DIdW17]{deWolffPositivstellensatz}
M.~Dressler, S.~Iliman, and T.~de~Wolff, \emph{A {P}ositivstellensatz for sums
  of nonnegative circuit polynomials}, SIAM Journal on Applied Algebra and
  Geometry \textbf{1} (2017), no.~1, 536--555.

\bibitem[DIdW19]{Dressler:Iliman:deWolff:FirstConstrained}
\bysame, \emph{An approach to constrained polynomial optimization via
  nonnegative circuit polynomials and geometric programming}, Journal of
  Symbolic Computation \textbf{91} (2019), 149--172.

\bibitem[DM22]{Dressler:Murray:AlgebraicPerspectives}
M.~Dressler and R.~Murray, \emph{Algebraic perspectives on signomial
  optimization}, SIAM Journal on Applied Algebra and Geometry \textbf{6}
  (2022), no.~4, 650--684.

\bibitem[{\'E}T89]{Erdi:Mathematical}
P.~{\'E}rdi and J.~T{\'o}th, \emph{Mathematical models of chemical reactions:
  theory and applications of deterministic and stochastic models}, Manchester
  University Press, 1989.

\bibitem[FdW22]{Forsgaard:deWolff:BoundarySONCCone}
Jens Forsg{\aa}rd and Timo de~Wolff, \emph{The algebraic boundary of the
  sonc-cone}, SIAM J. Appl. Algebra Geom. \textbf{6} (2022), no.~3, 468--502.

\bibitem[Fei95]{Feinberg:Existence}
M.~Feinberg, \emph{The existence and uniqueness of steady states for a class of
  chemical reaction networks}, Archive for Rational Mechanics and Analysis
  \textbf{132} (1995), no.~4, 311--370.

\bibitem[H{\etalchar{+}}67]{Hahn:Stability}
W.~Hahn et~al., \emph{Stability of motion}, vol. 138, Springer, 1967.

\bibitem[Hal09]{Hale:ODE}
J.K. Hale, \emph{Ordinary differential equations}, Courier Corporation, 2009.

\bibitem[HdW22]{Heuer:deWolff:DualityOfSONC}
J.~Heuer and T.~de~Wolff, \emph{The duality of {SONC}: Advances in
  circuit-based certificates}, 2022, arXiv:2204.03918.

\bibitem[HGGS18]{Hafstein:Lyapunov}
S.~Hafstein, S.~Gudmundsson, P.~Giesl, and E.~Scalas, \emph{Lyapunov function
  computation for autonomous linear stochastic differential equations using
  sum-of-squares programming}, Discrete \& Continuous Dynamical Systems-B
  \textbf{23} (2018), no.~2, 939.

\bibitem[IdW16a]{Iliman:deWolff:Circuits}
S.~Iliman and T.~de~Wolff, \emph{Amoebas, nonnegative polynomials and sums of
  squares supported on circuits}, Res. Math. Sci. \textbf{3} (2016), 3:9.

\bibitem[IdW16b]{Iliman:deWolff:FirstGP}
\bysame, \emph{Lower bounds for polynomials with simplex newton polytopes based
  on geometric programming}, SIAM Journal on Optimization \textbf{26} (2016),
  no.~2, 1128--1146.

\bibitem[JA15]{PapachristodoulouAdvancesSOS}
A.~Papachristodoulou J.~Anderson, \emph{Advances in computational lyapunov
  analysis using sum-of-squares programming}, Discrete \& Continuous Dynamical
  Systems - B \textbf{20} (2015), no.~1531-3492\_2015\_8\_2361, 2361.

\bibitem[Kha02]{KhalilNonlinearSystems}
H.K. Khalil, \emph{Nonlinear systems}, Pearson Education, Prentice Hall, 2002.

\bibitem[KNT21]{Katthaen:Naumann:Theobald:UnifiedFramework}
L.~Katth{\"a}n, H.~Naumann, and T.~Theobald, \emph{A unified framework of
  {SAGE} and {SONC} polynomials and its duality theory}, Mathematics of
  Computation \textbf{90} (2021), no.~329, 1297--1322.

\bibitem[Lag88]{Lagrange:Mecanique}
J.-L. Lagrange, \emph{M{\'e}canique analytique}, vol.~1, Gauthier-Villars,
  1788.

\bibitem[Las01]{Lasserre:GlobalOpt}
J.B. Lasserre, \emph{Global optimization with polynomials and the problem of
  moments}, SIAM Journal on Optimization \textbf{11} (2001), no.~3, 796--817.

\bibitem[Las15]{Lasserre:IntroductionPolynomialandSemiAlgebraicOptimization}
\bysame, \emph{An introduction to polynomial and semi-algebraic optimization},
  vol.~52, Cambridge University Press, 2015.

\bibitem[Lau09]{Laurent:Survey}
M.~Laurent, \emph{Sums of squares, moment matrices and optimization over
  polynomials}, Emerging Applications of Algebraic Geometry, IMA Vol. Math.
  Appl., vol. 149, Springer, New York, 2009, pp.~157--270.

\bibitem[LRS10]{DeLoera:et:al:Triangulations}
J.A.~De Loera, J.~Rambau, and F.~Santos, \emph{Triangulations}, Algorithms and
  Computation in Mathematics, vol.~25, Springer-Verlag, Berlin, 2010,
  Structures for algorithms and applications.

\bibitem[LWK{\etalchar{+}}22]{SumOfSquares}
B.~Legat, T.~Weisser, L.~Kapelevich, J.~Huchette, A.~Bhatia, T.~Votroubek,
  O.~Dowson, M.~Forets, A.~Moat, C.~Coffrin, C.~C., E.~Saba, E.~Hanson,
  G.~Berger, J.~TagBot, M.~Kaluba, M.~Lubin, S.~Tu, and T.~Kelman,
  \emph{jump-dev/sumofsquares.jl: v0.6.2}, March 2022.

\bibitem[Lya92]{Lyapunov:1992}
A.M. Lyapunov, \emph{The general problem of the stability of motion},
  International journal of control \textbf{55} (1992), no.~3, 531--534.

\bibitem[MCW21a]{Murray:Chandrasekaran:Wierman:NewtonPolytopes}
R.~Murray, V.~Chandrasekaran, and A.~Wierman, \emph{Newton polytopes and
  relative entropy optimization}, Foundations of Computational Mathematics
  \textbf{21} (2021), no.~6, 1703--1737.

\bibitem[MCW21b]{Murray:Chandrasekaran:Wierman:SigOptREP}
\bysame, \emph{Signomial and polynomial optimization via relative entropy and
  partial dualization}, Mathematical Programming Computation \textbf{13}
  (2021), no.~2, 257--295.

\bibitem[MNT22]{Murray:Naumann:Theobald:Sublinear}
R.~Murray, H.~Naumann, and T.~Theobald, \emph{Sublinear circuits and the
  constrained signomial nonnegativity problem}, Mathematical Programming
  (2022), 1--35.

\bibitem[Mot67]{Motzkin:AMGMIneq}
T.S. Motzkin, \emph{The arithmetic-geometric inequality}, Inequalities:
  Proceedings, Volume 1, Academic Press, 1967, pp.~203--224.

\bibitem[Par00a]{Parrilo:Thesis}
P.A. Parrilo, \emph{Structured semidefinite programs and semialgebraic geometry
  methods in robustness and optimization}, Ph.D. thesis, California Institute
  of Technology, 2000.

\bibitem[Par00b]{ParriloThesis}
\bysame, \emph{Structured semidefinite programs and semialgebraic geometry
  methods in robustness and optimization}, Ph.D. thesis, California Institute
  of Technology, 2000.

\bibitem[PDH07]{Park:Guidance}
S.~Park, J.~Deyst, and J.P. How, \emph{Performance and lyapunov stability of a
  nonlinear path following guidance method}, Journal of guidance, control, and
  dynamics \textbf{30} (2007), no.~6, 1718--1728.

\bibitem[PKC12]{pantea-jac}
C.~Pantea, H.~Koeppl, and G.~Craciun, \emph{Global injectivity and multiple
  equilibria in uni- and bi-molecular reaction networks}, Discrete Contin. Dyn.
  Syst. Ser. B \textbf{17} (2012), no.~6, 2153--2170.

\bibitem[PP02]{PapachristodoulouPrajna:LyapunovSOS}
A.~{Papachristodoulou} and S.~{Prajna}, \emph{On the construction of lyapunov
  functions using the sum of squares decomposition}, Proceedings of the 41st
  {IEEE} Conference on Decision and Control, 2002., vol.~3, 2002,
  pp.~3482--3487 vol.3.

\bibitem[Rez89]{Reznick:AGI}
B.~Reznick, \emph{{Forms Derived from the Arithmetic-Geometric Inequality}},
  {Mathematische Annalen} \textbf{{283}} (1989), {431--464}.

\bibitem[Ros95]{Roskam:Airplane}
J.~Roskam, \emph{Airplane flight dynamics and automatic flight controls},
  DARcorporation, 1995.

\bibitem[Sas99]{SastryNonlinearSystems}
S.~Sastry, \emph{Nonlinear systems: analysis, stability, and control}, Springer
  Science \& Business Media, 1999.

\bibitem[SdW19]{poem:software}
H.~Seidler and T.~de~Wolff, \emph{{POEM}: Effective methods in polynomial
  optimization, version 0.2.1.0(a)},
  \url{http://www.iaa.tu-bs.de/AppliedAlgebra/POEM/index.html}, July 2019.

\bibitem[SLX{\etalchar{+}}13]{She:etal:PolyLyap}
Z.~She, H.~Li, B.~Xue, Z.~Zheng, and B.~Xia, \emph{Discovering polynomial
  lyapunov functions for continuous dynamical systems}, Journal of Symbolic
  Computation \textbf{58} (2013), 41--63.

\bibitem[SS01]{Stamova:Lyapunov}
I.M. Stamova and G.T. Stamov, \emph{Lyapunov--razumikhin method for impulsive
  functional differential equations and applications to the population
  dynamics}, Journal of Computational and Applied Mathematics \textbf{130}
  (2001), no.~1-2, 163--171.

\bibitem[Tan06]{Tan:Nonlinear}
W.~Tan, \emph{Nonlinear control analysis and synthesis using sum-of-squares
  programming}, ProQuest, 2006.

\bibitem[Vid02]{VidyasagarNonlinearSystems}
M.~Vidyasagar, \emph{Nonlinear systems analysis}, SIAM, 2002.

\bibitem[Wan22]{Wang:supports}
J.~Wang, \emph{Nonnegative polynomials and circuit polynomials}, SIAM Journal
  on Applied Algebra and Geometry \textbf{6} (2022), no.~2, 111--133.

\bibitem[Wil97]{Williams:Chaos}
G.~Williams, \emph{Chaos theory tamed}, CRC Press, 1997.

\end{thebibliography}

\appendix
\section{Pseudocode}

We give the explicit pseudocode required for computing the sign distribution according to $\mathtt{distribute\_signs}$ here.
We assume that all constraints are of the form $\mathtt{lhs} \le \mathtt{rhs}$.

\begin{algorithm}[$\mathtt{distribute\_signs}$]~

	\noindent\INPUT{support sets $A, A^\prime \subset \N^n$ and variable coefficient vectors $\Vector{c} \in \R^A$ and $\Vector{d} \in \R^{A^\prime}$ of $P(\xb)$ and $P^\prime(\xb)$, set of existing constraints $\mathtt{constraints}$} \\
	\OUTPUT{set of sign constraints}
	\begin{algorithmic}[5]
		\State $A^+ \gets A \cap (2\N)^n$; $A^- \gets A \setminus A^+$
		\State $(A^\prime)^+ \gets A^\prime \cap (2\N)^n$; $(A^\prime)^- \gets A^\prime \setminus (A^\prime)^+$
		\State $\mathtt{hull\_P} \gets \vertices{\conv(A)}$
		\For{$\mathtt{exponent}$ in $\mathtt{hull\_P}$}
			\If{$\mathtt{exponent}$ in $A^-$}
				\State $A \gets A \setminus \set{\mathtt{exponent}}$, $A^- \gets A^- \setminus \set{\mathtt{exponent}}$
			\EndIf
		\EndFor
		\State $\mathtt{hull\_Pprime} \gets \vertices{\conv(A^\prime)}$
		\For{$\mathtt{exponent}$ in $\mathtt{hull\_Pprime}$}
			\If{$\mathtt{exponent}$ in $(A^\prime)^-$}
				\State $A^\prime \gets A^\prime \setminus \set{\mathtt{exponent}}$, $(A^\prime)^- \gets (A^\prime)^- \setminus \set{\mathtt{exponent}}$
			\EndIf
		\EndFor
		\For{$\alpb \in A^+$}
			\State $\mathtt{constraints} \mathrel{+}= [c_{\alpb} \ge 0]$ 
		\EndFor
		\For{$\alpb \in (A^\prime)^+$}
			\State $\mathtt{constraints} \mathrel{+}= [d_{\alpb} \ge 0]$ 
		\EndFor
		\State $\mathtt{constraints\_reduced} \gets \mathtt{constraints}[1]$
		\For{$\mathtt{con}$ in $\mathtt{constraints}[2: \mathtt{end}]$}
			\State $\mathtt{objective} \gets \mathtt{lhs}(\mathtt{con}) - \mathtt{rhs}(\mathtt{con})$
			\State $\mathtt{problem} = \mathtt{maximize}(\mathtt{objective}, \mathtt{constraints\_reduced})$
			\State $\mathtt{sol} = \mathtt{solve}(\mathtt{problem})$
			\If{$\mathtt{sol} \ge 1 \times 10^{-3}$}
				\State $\mathtt{constraints\_reduced} \mathrel{+}= \mathtt{con}$
			\EndIf 
		\EndFor

		\State \Return{$\mathtt{constraints\_reduced}$}
	\end{algorithmic}
\label{Alg:distributeSigns}
\end{algorithm}

\end{document}